\documentclass[a4paper,11pt]{amsart}
 \usepackage{float}
\usepackage{amssymb}
\usepackage{mathrsfs}
\usepackage{amsmath,amssymb,amsthm,latexsym,amscd,mathrsfs}
\usepackage{indentfirst}
\usepackage{stmaryrd}
\usepackage{graphicx}
\usepackage{extarrows}
\usepackage{multirow}
\usepackage{latexsym}
\usepackage{amsfonts}
\usepackage{color}
\usepackage{pictexwd,dcpic}
\usepackage{graphicx}
\usepackage{psfrag}
\usepackage{hyperref}
\usepackage{comment}
\usepackage{enumerate}

\usepackage{verbatim} %\begin{comment}... \end{comment}
\usepackage{setspace}

\numberwithin{equation}{section} \theoremstyle{plain}
\newtheorem{thm}{Theorem}[section]

\newtheorem{prop}[thm]{Proposition}
\newtheorem{lem}[thm]{Lemma}
\newtheorem{cor}[thm]{Corollary}

\newtheorem{conj}[thm]{Conjecture}

\newtheorem{rem}[thm]{Remark}

\newtheorem*{acknow}{Acknowledgments}
   
 \makeatletter

\def\<{\langle}
\def\>{\rangle}
\def\({\left(}
\def\){\right)}
\def\[{\left[}
\def\]{\right]}
\def\tr{\mathop{\text{tr}}}

\allowdisplaybreaks%????????
\makeatother
\title[A first eigenvalue estimate for embedded hypersurfaces]{A first eigenvalue estimate for embedded hypersurfaces in positive Ricci curvature manifolds}

\author[F.G. Li]{Fagui Li$^{}$}
\address{Frontier Interdisciplinary Domain, Beijing Institute of Technology, Zhuhai, Guangdong 519088, P. R. CHINA.}
\email{lifagui@bitzh.edu.cn}

\author[J. R. Yan]{Junrong Yan}
\address{%\color{red}
Department of Mathematics, Northeastern University, Boston, MA 02115, USA.}
\email{j.yan@northeastern.edu}

\date{}
\begin{document}
\maketitle

% % % % % % % % % %
\begin{abstract}
Let $\Sigma$ be a closed, embedded, oriented  hypersurface in a closed oriented   Riemannian
manifold $N$. Under a lower bound on the Ricci curvature and an upper
bound on the sectional curvature of $N$, we establish a lower bound for the
first nonzero eigenvalue of the Laplacian on $\Sigma$. The estimate depends on
the ambient curvature bounds, the normal injectivity radius, and the geometry of $\Sigma$
through its mean curvature and second fundamental form. This result extends the
classical eigenvalue estimate of Choi and Wang [J. Diff. Geom. \textbf{18} (1983), 559--562.] to the non-minimal case.

\end{abstract}

\section{Introduction}
The well-known Yau's Conjecture \cite{Schoen Richard and S. T. Yau 1994,Yau  1982} asserts that 
 \begin{conj}[\textbf{Yau's Conjecture \cite{Schoen Richard and S. T. Yau 1994,Yau  1982}}] \label{conj  Yau 1982}
If $ \Sigma$ is a closed   embedded minimal hypersurface of the unit sphere  $\mathbb{S}^{n+1}$, then the first nonzero eigenvalue of the Laplacian on $\Sigma$, denoted by $\lambda_1 (\Sigma)$, is  equal to $n$.
 
  \end{conj}
 In 1983, %In particular,
 Choi and Wang \cite{Choi Wang 1983} showed that $\lambda_1(\Sigma)\geq n/2$ and a careful argument (see \cite[ Theorem 5.1]{Brendle S 2013 survey of recent results}) implies that the strict inequality holds, i.e., $\lambda_1 (\Sigma)> n/2$. 
More precisely, Choi-Wang \cite{Choi Wang 1983} proved
\begin{thm}[\textbf{Choi-Wang} \cite{Choi Wang 1983}]\label{CW} 
Let $\Sigma$ be  a
closed, embedded, oriented minimal hypersurface  in a closed oriented  Riemannian manifold $N$ of dimension $n+1$. If the Ricci curvature of $N^{}$
is bounded from below by a positive constant $k>0$, then the first nonzero eigenvalue of the Laplacian on $\Sigma$, denoted by $\lambda_1(\Sigma)$,  has a lower bound given by  $\frac {k}2.$
  \end{thm}
Later, Choi and Schoen \cite{H I Choi and R Schoen 1985} were able to remove the orientability assumption of Theorem \ref{CW}.
   In addition,
    Tang and Yan \cite{TXY14,TY13} proved   Yau's Conjecture in the isoparametric case.
    Choe and Soret \cite{J Choe and M Soret  2009} were  able to verify the Yau's Conjecture for the Lawson surfaces and the Karcher-Pinkall-Sterling examples. Let $M$ be a three-dimensional manifold with positive Ricci curvature. Suppose there is no embedded minimal surface in $M$ which is the multiplicity 2 limit of a sequence of  constant mean curvature  (CMC) surfaces.
Recently,
   Sun \cite{Sun Ao 2020 Pacific J M} gave a lower bound for the first eigenvalue of CMC surfaces in $M^3$.
  For the minimal surface with genus $g$  in the unit sphere $\mathbb{S}^{3}$, 
Zhao \cite{Zhao Yuhang 2023}  proved that $\lambda_1\geq 1+\epsilon$, where $\epsilon$ is a positive constant depending only on $g$. For the minimal hypersurface in  $\mathbb{S}^{n+1}$ $(n\geq 2)$,  Duncan-Sire-Spruck \cite{Duncan  Sire and  Spruck IMRN 2024} and 
 Jim$\rm \acute{e}$nez-Chinchay-Zhou \cite{Jimenez Chinchay Zhou 2024} proved that $\lambda_1\geq n/2+\epsilon$, where $\epsilon$ is a
  specific function of $n$ and the squared length of the second fundamental form   of $\Sigma$, respectively.
Furthermore, Zhu
        \cite{Zhu Jonathan J 2017} provided a counterexample to Conjecture \ref{conj  Yau 1982} in the class of manifolds with boundary, and it implies that any resolution of Conjecture \ref{conj  Yau 1982}  would likely need to consider more geometric data than a Ricci curvature lower bound.
   For more details,  please see the elegant surveys by Choe \cite{J. Choe 2006} and Brendle \cite{Brendle S 2013 survey of recent results}.  
  
In this paper, we
%mainly deal with
focus on the non-minimal version 
of Theorem \ref{CW}.
Let
  $H=\tr_{g_{\Sigma}} h$ and
$S=|h|^2$ be  the mean curvature and the  squared length of the second fundamental form $h$ of $\Sigma$, respectively. Suppose
$$|H|_{\max}=\max_{\Sigma}|H|,\quad
S_{\max}=\max_{\Sigma}S,
$$
%$$
%\mathcal{C}=\sup_{ 0< t< \frac{1}{\sqrt{K}}   \arctan(\sqrt{K/ {S_{\max}}})} - \frac{nK\tan(\sqrt{K}t)+n\sqrt{S_{\max} K}}{\sqrt{K}-\sqrt{S_{\max}}\tan(\sqrt{K}t)}-\frac{1}{t},
%$$
%Choosing $\rho_a=\frac{1}{\sqrt{K}}
%  \arctan(t^{}\sqrt{K/ {S_{\max}}})$ $(0<t<1)$, one has
and 
 \begin{equation}\label{equation C constant}
 \mathcal{C}(r):=\sup_{t\in \left(0,r \right) } 
 \left[ -
  \frac{ntK+nS_{\max}}{(1-t)\sqrt{S_{\max}}}-\frac{\sqrt{K}}{\arctan(t^{}\sqrt{K/ {S_{\max}}})}\right] .
 \end{equation}
 One has
\begin{thm}\label{thm eigenvaule mean curvature}
Let $\Sigma$ be a closed, embedded, oriented %, non--totally geodesic
hypersurface in a closed oriented Riemannian manifold $N$ of dimension $n+1$.
Assume that the Ricci curvature of $N$ is bounded below by a positive constant $k>0$ and the sectional curvature of $N$ is bounded above by a positive constant $K>0$. 
Let $R>0$ denote the normal injectivity radius  of $\Sigma$ in $N$ (see
\eqref{equation rolling radius}).  Let $r=\min\{t_R,1\}$ and
$t_R=\sqrt{\frac{S_{\max}}{K}}\tan(R\sqrt{K})$.
%, where $t_R>0$ is defined by
%$$
%R=\frac{1}{\sqrt{K}}\arctan\!\bigl(t_R\sqrt{K/S_{\max}}\bigr).
%$$
Then the first nonzero eigenvalue of the Laplacian on $\Sigma$ satisfies
$$
\lambda_1(\Sigma)\ge
\frac{k}{2}
+\frac{|H|_{\max}}{2}
\left(\mathcal{C}(r)-\frac{n}{n+1}|H|_{\max}\right).
$$
\end{thm}
  \begin{rem}
If the mean curvature $H$ of $\Sigma$ is zero, the  lower bound of the first nonzero eigenvalue coincides with Theorem \ref{CW}. Thus, Theorem \ref{thm eigenvaule mean curvature} \footnote{
Theorem \ref{thm eigenvaule mean curvature} overlaps with that of \cite{Pak}.
We note that a key step (see Case B on page 275)  
 in \cite{Pak} implicitly requires additional assumptions, which are not stated explicitly.
The argument presented in this paper provides a complete and self-contained proof under the stated hypotheses. We thank P. T. Ho, the author of \cite{Pak}, for helpful correspondence on this point.
%Regarding the non-minimal case, a corresponding result can be found in  \cite{Pak}. However, it is regrettable that the proof contains an error, specifically in Case B on page 275. We sincerely appreciate the P. T. Ho, the author of \cite{Pak},  for communicating with us and confirming this point.
}  extends the result of  Theorem \ref{CW}  to non-minimal hypersurfaces.% In fact, the result will be better if the mean curvature is small enough (see Corollary \ref{cor eigenvaule mean curvature}).
  \end{rem}
%\begin{rem}
%A similar result appears in \cite{Pak}. However, it is regrettable that the proof contains an error, specifically in Case B on page 275.  We sincerely appreciate the author for communicating with us and confirming this point.
%    Our result overlaps with that of \cite{Pak}.
%We note that a key step \textcolor{red}{(JY: Could you please identify which step it is?)} in \cite{Pak} implicitly requires additional assumptions, which are not stated explicitly.
%The argument presented here provides a complete and self-contained proof under the stated hypotheses. We thank P. T. Ho, the author of \cite{Pak}, for helpful correspondence on this point.
%\end{rem}
    Due to the Cauchy inequality, we have $nS_{\max}\geq |H|_{\max}^2$. Hence
\begin{cor}
 Under the conditions of  Theorem \ref{thm eigenvaule mean curvature}, we have
 $$
 \lambda_1(\Sigma)\ge
 \frac{k}{2}
 +\frac{nS_{\max}}{2}
 \left(\mathcal{C}(r)-\frac{n^2}{n+1}S_{\max}\right).
 $$
  %In particular,
%   $$\lambda_1(\Sigma)\geq
%   \frac {k}{2}-
%   \frac{\sqrt{nS_{\max}}}{2}
%      \left(
%      nK+\frac{2\sqrt{S_{\max}K}}{\arctan(\sqrt{K})}+2n\sqrt{S_{\max}} +\frac{n^{\frac{3}{2}}}{n+1}\sqrt{S_{\max}}
 %     \right).$$
\end{cor}
%In fact, Theorem~\ref{thm eigenvaule mean curvature} follows from the more refined estimate stated below.

%By Theorem \ref{cor eigenvaule mean curvature}, we have the following corollaries. 
For the unit sphere $\mathbb{S}^{n+1}$, we have the following corollary by Theorem \ref{thm eigenvaule mean curvature}.  %one has

   \begin{cor}\label{cor sphere and volume}
    Let $\Sigma$  be a  closed  embedded mean-convex hypersurface in $\mathbb{S}^{n+1}$ with volume ${\rm Vol}(\Sigma)$.
    % Let    $r=\min\{\sqrt{ {S_{\max}}{}}\tan(R),1\}$     and
%    $ \mathcal{C}(r)=\sup_{t\in \left(0,r \right) } 
%        \left[ -
%         \frac{nt+nS_{\max}}{(1-t)\sqrt{S_{\max}}}-\frac{1}{\arctan(t^{}\sqrt{1/ {S_{\max}}})}\right] $.  
\begin{itemize}
 \item[(i)] The first nonzero eigenvalue of the Laplacian on $\Sigma$ satisfies
    $$
   \lambda_1(\Sigma)\geq
    \frac {n}{2}-
    \frac{|H|_{\max}}{2}
       \left(
       n+\frac{8+2n\pi}{\pi}\sqrt{S_{\max}}
       %+2n\sqrt{S_{\max}}
        +\frac{n}{n+1}|H|_{\max}
       \right).
  $$
 \item[(ii)]
    If ${\rm Vol}(\Sigma)\geq {\rm Vol}(\mathbb{S}^{n})$, then 
  $$
  S_{\max}\geq
  \frac{ \int_{\Sigma}S}{{\rm Vol}(\Sigma)}
   \geq
 \left[ 
   \frac {n}{2}-
   \frac{|H|_{\max}}{2}
      \left(
      n+\frac{8+2n\pi}{\pi}\sqrt{S_{\max}}
      %+2n\sqrt{S_{\max}}
       +\frac{n}{n+1}|H|_{\max}
      \right) \right] 
  \left[ 1- \frac{{\rm Vol}^2(\mathbb{S}^{n})}{{\rm Vol}^2(\Sigma)}\right] ,  $$
    where the equality holds if and only if $\Sigma$ is totally geodesic.
% \item[(iii)]      If ${\rm Vol}(\Sigma^n)\geq\left(1+\delta \right)  {\rm Vol}(\mathbb{S}^{n})$ $\left( \delta> 0\right)$, then   there is a positive constant $\eta$, depending only on $n$ and $\delta$, such that $ S_{\max}\geq\eta > 0$.
\end{itemize}
%  In particular, we have
%    $$
%S_{\max}
%    \geq
%  \left[ \frac {n}2+
%    \frac{|H|_{\max}}{2}
%       \left(\mathcal{C} -\frac{n}{n+1}|H|_{\max}\right) \right] 
%    \frac{{\rm Vol}^2(M^n)-{\rm Vol}^2(\mathbb{S}^{n})}{{\rm Vol}^2(M^n)}.
%    $$
     \end{cor}

\begin{rem}
If $\Sigma$ is a closed, immersed, non-totally geodesic, minimal hypersurface in  $\mathbb{S}^{n+1}$, % $\Sigma\looparrowright    \mathbb{S}^{n+1}$ is minimal and non-totally geodesic, 
then $ S_{\max}\geq n$ by  Simons' inequality \cite{Simons68}.  
Corollary \ref{cor sphere and volume} (ii)
can be compared with the Simons-type 
inequalities and pinching theorems
(cf.  \cite{Chern do Carmo Kobayashi 1970,M Scherfner S Weiss and  Yau 2012,Simons68}, etc.).
%Simons' inequality \cite{Simons68} shows that  $ S_{\max}\geq n$.
%Besides,  Cheng-Li-Yau \cite{Cheng Li Yau 1984 Heat equations}proved  that  there is  a positive constant $\delta$, depending only on $n$, such that %the area  of  $M^n$ satisfies
%   $
%   {\rm Vol}(\Sigma)\geq \left( 1+\delta\right) {\rm Vol}(\mathbb{S}^n).
%   $
%   Hence,   Corollary \ref{cor sphere and volume} (ii) also gives a weak positive lower bound for $S_{\max}$, i.e., $S_{\max}\geq\eta (n,\delta) > 0$.
%   Hence,   there is a positive constant $\eta$, depending only on $\delta$, such that $ S_{\max}\geq\eta > 0$ by Corollary \ref{cor sphere and volume} (ii). 
\end{rem}

%\begin{rem}
%Suppose $\Sigma$ is a   closed and embedded mean-convex hypersurface in   Duncan-Sire-Spruck (cf. Proposition 2.1 in \cite{Duncan  Sire and  Spruck IMRN 2024}) proved that the normal injectivity radius of $\Sigma$ satisfied $R > \arctan(\sqrt{1/ {S_{\max}}})-\varepsilon$ for any   positive number $\varepsilon$.
%\end{rem}
\section{Proof of the theorems}
\subsection*{Notational conventions}%*{Notations}
Let $(\Omega, g)$ be an $(n+1)$-dimensional ($n \geq 1$) compact Riemannian manifold with smooth boundary $\Sigma = \partial\Omega$, and let $g_{\Sigma}$ be the induced metric on $\Sigma$. We assume that the normal bundle of $\Sigma$ is trivial.
We use $\langle\cdot, \cdot\rangle$ to denote the inner product with respect to both $g$ and $g_{\Sigma}$ when no confusion occurs. We denote by $\nabla, \Delta$, ${\rm Hess}$ and $\mathrm{Ric}$ the gradient, the Laplacian, the Hessian and  the Ricci curvature tensor  on $\Omega$ respectively, while by $\nabla_{\Sigma}$ and $\Delta_{\Sigma}$ the gradient and the Laplacian on $\Sigma$ respectively. Let $\nu$ be the unit outward normal of $\Sigma$ and $f_\nu=\langle\nabla f, \nu\rangle$. 
We denote by $h(X, Y)=$ $g\left(\nabla_X \nu_, Y\right)$ and $H=\operatorname{tr}_{g_{\Sigma}} h$ the second fundamental form and the mean curvature of $\Sigma$
respectively. 

   Recall that the cut locus ${\rm Cut}(\Sigma)$ is defined to be the set of all cut points. A cut point is the first point on a normal geodesic initiating from the boundary $\Sigma$ at
    which this geodesic fails to minimize  uniquely for the distance function $\rho$. In other words, for $x\in \Sigma$ and  the arc-length parametrized geodesic 
     $\gamma_x(t)=\exp_x(-t\nu(x))$ 
     ($t\geq 0$), then
 $\gamma_x(t_0)\in {\rm Cut}(\Sigma)$ for
 \begin{equation*}
 t_0=t_0(x)=\sup\{t>0:\mathrm{dist}(\gamma_x(t),\Sigma)=t\}.
 \end{equation*}
 %{\color{blue} {\bf To be removed.} If $t_0(x)$ is finite, then $\exp_x(-t_0(x)\nu(x))$ is called a cut point of the boundary $\S$.}
 The set ${\rm Cut}(\Sigma)$ is known to have finite $(n-1)$-dimensional Hausdorff measure (cf. \cite{IT01, Li Yanyan and Louis Nirenberg 2005 CPAM}, etc.). 
   Define the distance function to the boundary $\Sigma$ by
  \begin{equation*}
  \rho=\rho_{\Sigma}(x)={\rm dist}(x,  \Sigma)=
  \inf_{y\in \Sigma}{\rm dist}(x,  y).
  \end{equation*}
  The distance function $\rho$ is smooth away from the cut locus $\mathrm{Cut}(\Sigma)$ of $\Sigma$. The distance between $\Sigma$ and $\mathrm{Cut}(\Sigma)$ is called the normal injectivity radius (also referred to as the rolling radius in \cite{Colbois Girouar  Hassannezhad jfa 2020}) of $\Sigma$ in $\Omega$:
\begin{equation}\label{equation rolling radius}
 %   {\rm NInj}
R_\Sigma (\Omega) = {\rm dist}(\Sigma,{\rm Cut} ({\Sigma})).
  \end{equation}
 Given
 $ h \in  (0,  %{\rm NInj}
 R_\Sigma (\Omega) ]$, 
 define the tubular neighbourhood 
 $\Omega_h =\{x\in \Omega:\rho_{\Sigma}(x)<h\}.$
    For each $x \in  \Omega_h$,
     there is exactly one nearest point to $x$ in $\Sigma$ and the exponential map $\gamma_x(t)=\exp_x(-t\nu(x))$ 
      defines a diffeomorphism between 
      $[0, h) \times \Sigma$
       and the tubular neighbourhood  $\Omega_h$.
     Moreover, for each $s \in [0, %{\rm NInj}
 R_\Sigma (\Omega) )$, the map $\gamma(s) : \Sigma  \rightarrow \Sigma_s$ is a diffeomorphism.

In this paper, we mainly use the following  formula by Reilly \cite{Reilly 1977}. For some generalizations of Reilly's formula and its applications, please refer to \cite{Li and Xia JDG 2019,Qiu and Xia IMRN 2015}, etc.
%We recall  Reilly's formula \cite{Reilly 1977}.
\begin{prop}[{Reilly's formula} \cite{Reilly 1977}]
Let $(\Omega,g)$ be an $(n+1)$-dimensional smooth compact connected Riemannian manifold with boundary $\Sigma$.
For a smooth function $f$  on ${\Omega}$, we have
$$
\begin{aligned}
& \int_{\Omega} (\Delta f)^2-\left|{\rm Hess} f\right|^2-\mathrm{Ric}\left(\nabla f,\nabla f \right)  \\
= & \int_{\Sigma} 2 f_\nu  \Delta_{\Sigma} f+H\left( f_\nu\right)^2+h\left(\nabla_{\Sigma} f, \nabla_{\Sigma} f\right). \\
\end{aligned}
$$
%where $\Delta$ and $\n$ are the Laplacian operator and the Levi-Civita connection of the Riemannian metric $g$ respectively, $\Delta_\S$ and $\nabla_\S$ the corresponding ones of the induced Riemannian metric $g_\S$ respectively,  $H$ and $h$ are the mean curvature and the second fundamental form of the boundary respectively, $\mathrm{Ric}_g$ is the Ricci curvature tensor of the Riemannian metric $g$, and $\nu$ is the unit outer normal along the boundary.
\end{prop}   
   In the subsequent proof, we mainly need the following theorem by Kasue \cite{Atsushi Kasue 1984,Atsushi Kasue 1982 JapanMath}.

 \begin{thm}[Theorem 0.3 in \cite{Atsushi Kasue 1984}]\label{Thm Atsushi Kasue 1982 84}
Let $M$ be a closed submanifold of a Riemannian manifold $N$  and $x$ a point of $N_0\backslash M$, where $N_0$ is the interior of $N$ ($N_0=N$ if $\partial N=\emptyset$).
%the boundary of $N$ is empty). 
Suppose there exists a distance minimizing geodesic $\sigma:[0, a] \rightarrow N$ from $M$ through $x=\sigma\left(a^{\prime}\right)\left(a^{\prime}<a\right)$, (so that $\rho_M={\rm dist}(x,  M)$ is smooth near $x$). Let $\mathscr{K}$ be a continuous function on $\left[0, a^{\prime}\right]$ such that the sectional curvature of any tangent plane containing $\dot{\sigma}(t)$ is bounded from above by $\mathscr{K}(t)$ $\left(t \in\left[0, a^{\prime}\right]\right)$; in the case when $\operatorname{dim} M>0$, let $\Gamma$ be a real number such that all the eigenvalues of the second fundamental form $S_{\dot{\sigma}(0)}$ of $M$ are bounded from below by $\Gamma$. 
Let $h_{\mathscr{K}, \Gamma}$ and $f_\mathscr{K}$ be the solution of 
(\ref{equation h-mathscr-Gamma}) and 
(\ref{equation f-mathscr}), respectively.
%Let $h_{\mathscr{K}, \Gamma}$ (resp. $f_\mathscr{K}$) be the solution of equation (0.4) defined by $\mathscr{K}$ and $\Gamma$ (resp. the solution of equation (0.2) defined by $\mathscr{K}$). 
\begin{equation}\label{equation h-mathscr-Gamma}
h''_{\mathscr{K}, \Gamma}(t)+
\mathscr{K}(t)h_{\mathscr{K}, \Gamma}(t)=0,
\quad h_{\mathscr{K}, \Gamma}(0)=1\ and \  
h'_{\mathscr{K}, \Gamma}(0)=\Gamma.
\end{equation}
\begin{equation}\label{equation f-mathscr}
f''_\mathscr{K}(t)+\mathscr{K}(t)f_\mathscr{K}(t)=0,
\quad f_\mathscr{K}(0)=0\ and\  f'_\mathscr{K}(0)=1.
\end{equation}
Suppose $h_{\mathscr{K}, \Gamma}$ is positive on $\left[0, a^{\prime}\right]$. Then the Hessian  of $\rho_M$ has an estimate:
 $$
\left({\rm Hess} \rho_M\right)_x(V, V) \geqq\left(\log h_{\mathscr{K}, \Gamma}\right)^{\prime}\left(a^{\prime}\right)\left\{\|V\|^2-\left\langle\dot{\sigma}\left(a^{\prime}\right), V\right\rangle^2\right\}
 $$
for any $V \in T_xN$, and in addition, if +%\textcolor{red}{
$V \in d\left(\exp _M\right)_{\dot{\sigma}(a^{\prime})}(\Pi)$,%JY: should be $\exp_M$?},
 $$
\left({\rm Hess} \rho_M\right)_x(V, V) \geqq\left(\log f_\mathscr{K}\right)^{\prime}\left(a^{\prime}\right)\left\{\|V\|^2-\left\langle\dot{\sigma}\left(a^{\prime}\right), V\right\rangle^2\right\}
 $$
 where $\Pi$ denotes the vertical subspace in the tangent space at $\dot{\sigma}(a^{\prime})$ of the normal bundle $\nu(M)$ for $M$. (We take $d\left(\exp _M\right)_{\dot{\sigma}(a^{\prime})}(\Pi)=T_xN$ if $\operatorname{dim} M=0$.) In particular,
 $$
 \Delta \rho_M(x) \geqq\left\{\dim(M)\left(\log h_{\mathscr{K}, \Gamma}\right)^{\prime}+\big(\dim(N)-\dim(M)-1\big)\left(\log f_\mathscr{K}\right)^{\prime}\right\}\left(a^{\prime}\right)
 $$
 \end{thm}

 By Theorem \ref{Thm Atsushi Kasue 1982 84}, we obtain the following result.%one has
  \begin{lem}\label{lemma laplace pho lower bound}
 Let $\Sigma$ be  a closed, embedded, oriented  hypersurface  in a closed oriented  Riemannian manifold $N$ of dimension $n+1.$
  If the sectional curvature of $N^{}$ is bounded from above  by a positive constant $K>0$,
% If the Ricci curvature of $N^{n+1}$ is bounded from below by a positive constant $k=nK>0$, 
 then there exists some function $C_{S,K}(\rho)$ %depending only on the norm of second fundamental form $S$, $\rho$  and $K$, 
 such that  the distance function $\rho(x)={\rm dist}(x,  \Sigma)$  satisfies $$  \Delta \rho\geq C_{S,K}(\rho)=-
 \frac{nK\tan(\sqrt{K}\rho)+n\sqrt{S_{\max} K}}{\sqrt{K}-\sqrt{S_{\max}}\tan(\sqrt{K}\rho)},
 $$ 
 where $S_{\max}=\sup_{p\in\Sigma}S(p)$ and
 $0\leq \rho< \frac{1}{\sqrt{K}}
 \arctan(\sqrt{K/ {S_{\max}}})$.
 \end{lem}
 \begin{proof}
Let $\mathscr{K}=K$ and $\Gamma=-\sqrt{S_{\max}}$ in (\ref{equation h-mathscr-Gamma}) of Theorem \ref{Thm Atsushi Kasue 1982 84}, we have
$$
h''_{{K}, -\sqrt{S_{\max}}}(\rho)+Kh_{{K}, -\sqrt{S_{\max}}}(\rho)=0,$$
where
$ h_{{K}, -\sqrt{S_{\max}}}(0)=1$  and
$ h'_{{K}, -\sqrt{S_{\max}}}(0)=-\sqrt{S_{\max}}.
$
Then
 \begin{equation*}
 h_{{K}, -\sqrt{S_{\max}}}(\rho)=
 \cos(\sqrt{K}\rho)
 -\sqrt{S_{\max}/K} \sin(\sqrt{K}\rho),
 \end{equation*}
 and 
  \begin{equation*}
  \Delta \rho\geq C_{S,K}(\rho)=n\frac{h'_{{K}, -\sqrt{S_{\max}}}(\rho)}{ h_{{K}, -\sqrt{S_{\max}}}(\rho)}
  =-
   \frac{nK\tan(\sqrt{K}\rho)+n\sqrt{S_{\max} K}}{\sqrt{K}-\sqrt{S_{\max}}\tan(\sqrt{K}\rho)}.
  \end{equation*}
 \end{proof}

\begin{proof}[\textbf{Proof of Theorem $\mathbf{\ref{thm eigenvaule mean curvature}}$}] 
Due to Theorem \ref{CW}, we assume in the following proof that the mean curvature is nonzero.
Note that the first Betti number of $N$ must be zero since the Ricci curvature of $N$ is strictly positive. Combining this with the fact that both $\Sigma$ and $N$ are orientable, by looking at the exact sequences of homology groups, one can see that $\Sigma$ divides $N$ into two components $\Omega_1$ and $\Omega_2$ such that $\partial \Omega_1=\partial \Omega_2=\Sigma$.
 Let $u$ be the first eigenfunction of $\Sigma$, that is,
 $$
{\Delta} u+\lambda_1 u=0.
 $$
 Let $\nu$ be the unit outward normal of $\Sigma$ in $\Omega=\Omega_1$.  We denote by $h(X, Y)=$ $g\left(\nabla_X \nu_, Y\right)$ and $H=\operatorname{tr}_{g_{\Sigma}} h$ the second fundamental form and the mean curvature of $\Sigma$.
 Without loss of generality,  we  assume that
 $$\int_{\Sigma}h\left(\nabla_{\Sigma} f, \nabla_{\Sigma} f\right)\geq0.$$
% Suppose
 % \begin{equation*}
%  \rho=\rho(x)={\rm dist}(x,  \Sigma).
 % \end{equation*}
 Let $f$ be the solution of the Dirichlet problem such that:
\begin{equation}\label{equation Deltaf0}
 \begin{cases}
 \Delta f=0, & \text { in } \Omega; 
 \\ f=u, & \text { in } \partial \Omega=\Sigma.
 \end{cases}
\end{equation}
 Then $f$ is a function defined on $\Omega$ smooth up to $\partial \Omega$. % and we have   
 By Reilly's formula, we have
\begin{equation}\label{equation Reilly f manf}
 \int_{\Omega} (\Delta f)^2-\left|{\rm Hess} f\right|^2-\mathrm{Ric}\left(\nabla f,\nabla f \right)  
\geq  \int_{\Sigma} 2 f_\nu  \Delta_{\Sigma} f+H\left( f_\nu\right)^2. 
\end{equation}
 Due to (\ref{equation Deltaf0}), one has
\begin{equation}\label{equation fmu}
- \int_{\Sigma}  f_\nu  \Delta_{\Sigma} f=
 \lambda_1
 \int_{\Sigma}  f_\nu f=
  \lambda_1
  \int_{\Omega} f  \Delta f+|\nabla f|^2= 
   \lambda_1 \int_{\Omega} |\nabla f|^2.
\end{equation}
Next, we choose the normal coordinate $\left\lbrace e_1,e_2,\cdots,e_{n+1}\right\rbrace $ at the point $p\in \Omega$ such that $\nabla f(p)=f_1(p)=\left\langle \nabla f,e_1\right\rangle e_1(p)$. At $p\in \Omega$, one has
$$
\nabla_{e_i}|\nabla f|^2=\sum_{j}\nabla_{e_i}f_j^2=2\sum_{j}f_jf_{ij}=2f_1f_{1i},
$$
and %by $\Delta f=0$ in $\Omega$ we have
$$
|\nabla|\nabla f|^2|^2=\sum_{i=1}^{n+1}|\nabla_{e_i}|\nabla f|^2|^2
=4\sum_{i=1}^{n+1}f_1^2f_{1i}^2=4|\nabla f|^2\sum_{i=1}^{n+1}f_{1i}^2.
$$
Since $\Delta f=0$ in $\Omega$, we have
\begin{align*}
2\sum_{i=1}^{n+1}f_{1i}^2
&=\frac{2n}{n+1}f_{11}^2+
\frac{2}{n+1}
f_{11}^2+
\sum_{i=2}^{n+1}\left( f_{1i}^2+f_{i1}^2\right) \\
&=\frac{2n}{n+1}f_{11}^2+
\frac{2}{n+1}
\left( \sum_{i=2}^{n+1}f_{ii}\right) ^2+
\sum_{i=2}^{n+1}\left( f_{1i}^2+f_{i1}^2\right)\\
&\leq \frac{2n}{n+1}\sum_{i=1}^{n+1}f_{ii}^2+\sum_{i=2}^{n+1}\left( f_{1i}^2+f_{i1}^2\right)\\
&\leq\frac{2n}{n+1}\sum_{i,j=1}^{n+1}f_{ij}^2=\frac{2n}{n+1}|{\rm Hess}f|^2.
\end{align*}
Hence 
\begin{equation}\label{equation Hessf}
|\nabla|\nabla f|^2|^2\leq
\frac{4n}{n+1}
|\nabla f|^2
|{\rm Hess}f|^2.
\end{equation}
Suppose  $\rho_a=\frac{1}{\sqrt{K}}
      \arctan(t\sqrt{K/ {S_{\max}}})$ is  a fixed real number, where
      $0<t<r=\min \{t_R, 1 \}$ and   $R=\frac{1}{\sqrt{K}}
            \arctan(t_R\sqrt{K/ {S_{\max}}})$.
            % one has $ 0< \rho_a< \frac{1}{\sqrt{K}}    \arctan(\sqrt{K/ {S_{\max}}})$. 
Let $\eta $ be a cutoff function defined by
\begin{equation*}
\eta(\rho)=
\left\{
\begin{aligned}
  0,&\quad \rho>\rho_a;\\
 1- \rho/\rho_a,&\quad 0\leq \rho\leq \rho_a.%;%\\
%  1,&\quad x<0.
\end{aligned}
\right.
\end{equation*}
By Lemma \ref{lemma laplace pho lower bound} and (\ref{equation Hessf}), we have
\begin{align}\label{equation H mean curvature}
 \int_{\Sigma}  H\left( f_\nu\right)^2
 &\geq
 -|H|_{\max}\int_{\Sigma}  \left( f_\nu\right)^2 \geq
  -|H|_{\max}\int_{\Sigma} |\nabla f|^2\\
  &=
    |H|_{\max}\int_{\Omega} {\rm div}\left(\eta  |\nabla f|^2\nabla \rho\right)  \notag\\
 &=|H|_{\max}\int_{\Omega} \eta|\nabla f|^2\Delta \rho +
 \eta\left\langle \nabla |\nabla f|^2 ,\nabla\rho\right\rangle+
  |\nabla f|^2\left\langle \nabla \eta ,\nabla\rho\right\rangle\notag\\
 &\geq |H|_{\max}\int_{\Omega}
\eta\left(   C_{S,K}-\frac{1}{\rho_a}\right)  |\nabla f|^2 
 -\eta|\nabla|\nabla f|^2|\notag\\
 &\geq |H|_{\max}\int_{\Omega}\eta\left(    C_{S,K}-\frac{1}{\rho_a}\right)  |\nabla f|^2 -2\sqrt{\frac{n}{n+1}}|\nabla f ||{\rm Hess}f| \eta\notag\\
 &\geq\int_{\Omega}
\eta |H|_{\max}\left(    C_{S,K}-\frac{1}{\rho_a} -\frac{n}{n+1}  |H|_{\max}\right) |\nabla f|^2  -
\eta|{\rm Hess}f|^2 \notag\\
 &\geq\int_{\Omega}
 |H|_{\max}\left(   C_{S,K}(\rho_a)-\frac{1}{\rho_a} -\frac{n}{n+1}|H|_{\max}\right) |\nabla f|^2  -
|{\rm Hess}f|^2 \notag.
\end{align}
By (\ref{equation Reilly f manf}), (\ref{equation fmu}) and (\ref{equation H mean curvature}), we have
$$
\left[ 2\lambda_1-k-
   |H|_{\max}
   \left(   C_{S,K}(\rho_a)-\frac{1}{\rho_a} -\frac{n}{n+1}|H|_{\max}\right)
 \right] 
 \int_{\Omega} |\nabla f|^2\geq 0.
$$
Hence
$$
2\lambda_1\geq k+
   |H|_{\max}
   \left( C_{S,K}(\rho_a)-\frac{1}{\rho_a} -\frac{n}{n+1}|H|_{\max}\right), 
$$
for all $ 0< \rho_a< \frac{1}{\sqrt{K}}
  \arctan(r\sqrt{K/ {S_{\max}}})$.
%  Choosing $\rho_a=\frac{1}{\sqrt{K}}     \arctan(t^{}\sqrt{K/ {S_{\max}}})$ $(0<t<r)$, one has
Let
 $$
  \begin{aligned}
F_{S,K}(\rho_a)
&:=C_{S,K}(\rho_a)-\frac{1}{\rho_a}\\
&=
-\frac{nK\sqrt{1/S_{\max}}t+n\sqrt{S_{\max}}}{1-t}-
\frac{\sqrt{K}}{ \arctan(t\sqrt{K/ {S_{\max}}})}\\
&=-\frac{nKt+nS_{\max}}{(1-t)\sqrt{S_{\max}}}-
\frac{\sqrt{K}}{ \arctan(t\sqrt{K/ {S_{\max}}})},
 \end{aligned}
 $$
one has
    $$\lambda_1(\Sigma)\geq\frac {k}2+
    \frac{|H|_{\max}}{2}
       \left(   \mathcal{C}(r) -\frac{n}{n+1}|H|_{\max}\right),$$
       where
  $$
  \mathcal{C}(r)=\sup_{\rho_a\in \left( 0, 
   \frac{1}{\sqrt{K}}
    \arctan(r\sqrt{K/ {S_{\max}}})\right) }
  F_{S,K}(\rho_a)=\sup_{t\in \left(0,r \right) } 
  \left[ -
   \frac{nKt+nS_{\max}}{(1-t)\sqrt{S_{\max}}}-\frac{\sqrt{K}}{\arctan(t^{}\sqrt{K/ {S_{\max}}})}\right] .
  $$
This completes the proof.
 \end{proof}
% \begin{rem}
 %Under the assumptions of  Theorem \ref{thm eigenvaule mean curvature}, we have
% $$\lambda_1(\Sigma)\geq\frac {k}2+ \frac{|H|_{\max}}{2}   \left(   -\frac{\sqrt{K}}{ \arctan(    \frac{1}{2}    \sqrt{K/ {S_{\max}}})}-2n\sqrt{S_{\max}}-       nK\sqrt{1/S_{\max}}  -\frac{n}{n+1}|H|_{\max}\right),$$by  choosing  $t= \frac{1}{2}$ in \eqref{equation C constant}   \end{rem}

Since the definition of  $ \mathcal{C}(r)$ is rather complex, we consider the special case where $t_R = \frac{1}{2}$, yielding the following simpler and clearer expression:
\begin{cor}\label{cor eigenvaule mean curvature}
%Let $\Sigma^n$ be  a compact, embedded, oriented  hypersurface   in a compact oriented  Riemannian manifold $N^{n+1}$. If  the Ricci curvature of $N^{}$
%is bounded from below by a positive constant $k>0$, the sectional curvature of $N^{}$ is bounded from above  by a positive constant $K>0$ and 
 Under the conditions of  Theorem \ref{thm eigenvaule mean curvature},
if the
normal injectivity radius  (see (\ref{equation rolling radius}))  $R \geq \frac{1}{\sqrt{K}}
      \arctan(\frac{1}{2}\sqrt{K/ {S_{\max}}})$, then 
 $$\lambda_1(\Sigma)\geq
 \frac {k}{2}-
 \frac{|H|_{\max}}{2}
    \left(
    nK+\frac{2\sqrt{S_{\max}K}}{\arctan(\sqrt{K})}+2n\sqrt{S_{\max}} +\frac{n}{n+1}|H|_{\max}
    \right).$$
  \end{cor}
  
\begin{proof}
%[\textbf{Proof of Theorem $\mathbf{\ref{cor eigenvaule mean curvature}}$}]
Without loss of generality,
we suppose  $\Sigma$ is non-totally geodesic.
By Theorem \ref{thm eigenvaule mean curvature} one has
$$\lambda_1(\Sigma)\geq
\frac {k}{2}+
\frac{|H|_{\max}}{2}
   \left(\mathcal{C}(r) -\frac{n}{n+1}|H|_{\max}\right),$$
   where $\mathcal{C}(r)$ (see \eqref{equation C constant}) is a non-positive   constant, depending  on $n,r,K,S_{\max}.$
    Therefore, it suffices to estimate the lower bound of $\mathcal{C}(r)$. Next, consider the following two cases.\\
\textbf{Case $({\rm i})$.}  If $S_{\max}\leq 1$, then we choose $t= s\sqrt{S_{\max}}\leq \frac{1}{2}$  $(0<s<1)$,  by  \eqref{equation C constant} we have
\begin{equation*}
  \begin{aligned}
\mathcal{C}(r)
&\geq 
-
 \frac{ntK+nS_{\max}}{(1-t)\sqrt{S_{\max}}}-\frac{\sqrt{K}}{\arctan(t^{}\sqrt{K/ {S_{\max}}})}
 =  -\frac{nsK+n\sqrt{S_{\max}}}{1-s\sqrt{S_{\max}}}-\frac{\sqrt{K}}{\arctan(s^{}\sqrt{K})}\\
 &\geq 
 -\frac{nsK+n\sqrt{S_{\max}}}{1-s}-\frac{\sqrt{K}}{\arctan(\sqrt{K})s}\\
 &=nK -\frac{nK+n\sqrt{S_{\max}}}{1-s}-\frac{\sqrt{K}}{\arctan(\sqrt{K})s},
 \end{aligned}
\end{equation*}
the last inequality follows from the convexity of the function $\arctan(s)$ $(s>0)$.
Choosing $s_0\in(0,1)$  such that 
 $$
 \frac{nK+n\sqrt{S_{\max}}}{1-s_0\sqrt{S_{\max}}}
 =\frac{\sqrt{K}}{\arctan(\sqrt{K})s_0},
$$
i.e.,
$s_0=\left[ \sqrt{S_{\max}}+n\left(\sqrt{S_{\max}/K}+
\sqrt{K} \right) \arctan(\sqrt{K})\right] ^{-1}$. Then we have $t= s_0\sqrt{S_{\max}}\leq \frac{1}{2}$ and 
$$
\mathcal{C}(r)\geq nK -\frac{2\sqrt{K}}{\arctan(\sqrt{K})s_0}
=-nK-\frac{2\sqrt{S_{\max}K}}{\arctan(\sqrt{K})}-2n\sqrt{S_{\max}}.
$$
\textbf{Case $({\rm ii})$.}  If $S_{\max}\geq 1$, then we choose  $t= \frac{1}{2}$ and by  \eqref{equation C constant}   one has
\begin{equation*}
  \begin{aligned}
\mathcal{C}(r)
&\geq -
      nK\sqrt{1/S_{\max}}
 -\frac{\sqrt{K}}{ \arctan(
   \frac{1}{2}
   \sqrt{K/ {S_{\max}}})}-2n\sqrt{S_{\max}}\\
 &\geq 
 -nK\sqrt{1/S_{\max}}
  -\frac{2\sqrt{K{S_{\max}}}}{ \arctan(\sqrt{K})}-2n\sqrt{S_{\max}}\\
   &\geq 
   -nK
    -\frac{2\sqrt{K{S_{\max}}}}{ \arctan(\sqrt{K})}-2n\sqrt{S_{\max}}.
\end{aligned}
\end{equation*}
 This completes the proof.
 \end{proof}

 \begin{lem}[Proposition 2.1 in \cite{Duncan  Sire and  Spruck IMRN 2024}]\label{lem DS mean-convex imrn2024}
 Suppose $\Sigma$ is a smooth,  closed, and embedded mean-convex hypersurface in  $\mathbb{S}^{n+1}$.
Then $\Sigma_t=\left\lbrace \exp_x(-t\nu(x))\in \mathbb{S}^{n+1}: x\in \Sigma \right\rbrace $  is a smooth, closed, and embedded strictly mean-convex hypersurface in $\mathbb{S}^{n+1}$ for  $|t| \in \left(0, \arctan(\sqrt{1/ {S_{\max}}})\right) $.
 \end{lem}
   \begin{proof}[\textbf{Proof of  Corollary $\mathbf{\ref{cor sphere and volume}}$}]
  For the unit sphere $\mathbb{S}^{n+1}$,   we have $k=n$ and $K=1$.  By Ge-Li (cf. Theorem 1.6 in  \cite{Ge Li 2021 Perdomo conjecture}), we have
      $$
S_{\max}{\rm Vol}(\Sigma)\geq    \int_{\Sigma}S
    \geq
    \lambda_1(\Sigma)
    \frac{{\rm Vol}^2(\Sigma)-{\rm Vol}^2(\mathbb{S}^{n})}{{\rm Vol}(\Sigma)},
    $$
    where the equality holds if and only if $\Sigma$ is totally geodesic.
       By Lemma \ref{lem DS mean-convex imrn2024},    the normal injectivity radius of $\Sigma$ satisfies that
       $$R \geq \frac{1}{\sqrt{K}}
             \arctan(\frac{1}{2}\sqrt{K/ {S_{\max}}}).$$
                 Combining the inequalities above, %$nS_{\max}\geq |H|_{\max}^2$ and Theorem  \ref{thm eigenvaule mean curvature}, 
    we complete the proof by  Corollary \ref{cor eigenvaule mean curvature}. 
     \end{proof} 
     \begin{acknow}
     The authors are deeply grateful to the anonymous referee for their valuable comments and careful review. Their thoughtful feedback has been instrumental in strengthening the paper.
     %The authors thank the anonymous referee for their valuable suggestions.
     \end{acknow}
\singlespacing
{\noindent \bf  Funding} F. G. Li is partially supported by  NSFC (No. 12171037, 12271040, 12501061) and Research Start up Funding of Beijing Institute of Technology (No. 5640011253301). J. R. Yan is partially supported by Zelevinsky Postdoctoral Fellowship.
\singlespacing
{\noindent \bf  Data Availability Statement} This manuscript has no associated data.
    \singlespacing
{\noindent \bf  Declarations} 

    {\noindent\bf Conflict of interest}
    The authors declare that they have no conflict of interest.

\end{document}